\documentclass[reqno,b6pape]{amsart}
\usepackage{amsmath}
\usepackage{amssymb}
\usepackage{amsthm}
\usepackage{enumerate}
\usepackage[mathscr]{eucal}
\usepackage{eqlist}
\usepackage{hyperref}
\usepackage[polish,english]{babel}
\usepackage{polski}
\hypersetup{colorlinks=true,linkcolor=red, anchorcolor=green, citecolor=blue, urlcolor=red, filecolor=magenta, pdftoolbar=true}

\theoremstyle{plain}
\newtheorem{theorem}{Theorem}[section]
\newtheorem{prop}[theorem]{Proposition}
\newtheorem{lemma}{Lemma}[section]
\newtheorem{corol}{Corollary}[section]

\theoremstyle{definition}
\newtheorem{definition}{Definition}[section]
\newtheorem{remark}{\textnormal{\textbf{Remark}}}
\newtheorem{example}{Example}
\newtheorem{note}{Note}
\newtheorem{question}{Question}

\theoremstyle{remark}

\numberwithin{equation}{section}

\begin{document}
\setcounter{page}{1}
\title[Fixed point theorem in quasi-$(2,\beta)$-Banach spaces and  its applications ]{A new fixed point approach  to hyperstability of radical-type functional equations in quasi-$(2,\beta)$-Banach spaces}
\author[Iz. EL-Fassi ]{Iz-iddine EL-Fassi$^{\star 1}$   }

\newcommand{\acr}{\newline\indent}

\address{$^{1}$Department of Mathematics,\acr
                   Faculty of Sciences and Techniques,\acr
                   Sidi Mohamed Ben Abdellah University, \acr
                   B.P. 2202, Fez,
                   Morocco.}				
									\email{$^{1}$E-mail: izidd-math@hotmail.fr; izelfassi.math@gmail.com; iziddine.elfassi@usmba.ac.ma}  \renewcommand{\thefootnote}{} \footnotetext{ $^{\star}$Corresponding author.}

\begin{abstract} 
The main focus of this paper is to   define the notion of \textit{quasi-$(2,\beta)$-Banach space}  and show some properties in this new space, by help of it and under some  natural assumptions, we  prove that the fixed point theorem \cite[Theorem 2.1]{dang} is still valid in the setting of quasi-$(2,\beta)$-Banach spaces, this is also an extension of the fixed point result of Brzd\k{e}k et al. \cite[Theorem 1]{krs} in $2$-Banach spaces to quasi-$(2,\beta)$-Banach spaces. In the next part, we give a general solution of the  radical-type functional equation \eqref{eq0}.  In addition, we study the hyperstability  results for these functional equation by applying the aforementioned fixed point theorem, and at the end of this paper we will derive some consequences.
\end{abstract}
\maketitle

\section{Introduction}
Let $E$ and $F$ be linear spaces over a real or complex scalar field $\mathbb{K}$. We recall that a function $g : E \to F $ satisfies  the general quadratic equation if
\begin{align}\label{eq00}
g\left(ax+by\right)+g\left(ax-by\right)=cg(x)+dg(y),\;\;x,y\in E,
\end{align}
where $a,b,c,d\in \mathbb{K}\setminus\{0\}.$ are fixed numbers. We see that for $c=d=a^2+b^2$ in \eqref{eq00} we get the Euler-Lagrange functional equation investigated by J.M. Rassias \cite{j1,j2} (see also \cite{pi}), while the quadratic functional equation corresponds to $a=b=1$ and $c=d=2.$

\vspace*{2mm}
In the sequel, $\mathbb{N}$, $\mathbb{R}$ and $\mathbb{K}$ denote the sets of all positive integers, real numbers and the field of real or complex numbers, respectively.  Also, $W^V$ denotes the set of all functions from a set $V\neq\emptyset$ to a set $W\neq\emptyset$. We put  $\mathbb{R}_0:=\mathbb{R}\setminus \{0\}$,  $\mathbb{R}_+:=[0,\infty)$ and $\mathbb{N}_0:=\mathbb{N}\cup \{0\}$.
\vspace*{2mm}

First of all, let us recall the history in the stability theory for functional equations. The story of the stability of functional equations dates back to 1925 when a stability result appeared in the celebrated book by P\'{o}olya and Szeg \cite{sz}. In 1940, Ulam \cite{1} posed the famous Ulam stability problem which was partially solved by Hyers \cite{2} in the framework of Banach spaces. Later, Hyers' result was extended by Aoki \cite{0} and next by  Rassias \cite{3}. Since then numerous papers on this subject have been published and we refer to \cite{agr,ci,cz,izdin,hy,ham,ju,mo} for more details. On the other hand, fixed point theorems have been already applied in the theory of Hyers-Ulam stability by several authors (see for instance \cite{cad,fas1,jung,jung0}) and
it seems that Baker (see \cite{baker}) has used this tool for the first time in this field.

According to our best knowledge, the first hyperstability result was published in \cite{Bo}, and concerned
ring homomorphisms. However, it seems that the term \textit{hyperstability} was used for the ﬁrst time in \cite{ma} (quite often it is confused with \textit{superstability}, which admits also bounded functions). There are many researchers investigating the hyperstability results for functional equations in many areas (see, e.g., \cite{18,180,dang,9,19}). 

The radical functional equation is one of the popular topics for investigating in the theory of stability. Nowadays, a lot of papers concerning the stability and the hyperstability of the radical functional equation
in various spaces was appeared (see in \cite{aiem,ali,fass,izdin,i0,ham} and references therein).

In \cite{ham}, Khodaei et al.  solved the following functional equations
$$f\left(\sqrt{ax^{2}+by^{2}}\right)=af(x)+bf(y),\;\;x,y \in \mathbb{R},
$$
$$f\left(\sqrt{ax^{2}+by^{2}}\right)+f\left(\sqrt{|ax^{2}-by^{2}|}\right)=a^2 f(x)+b^2 f(y),\;\;x,y \in \mathbb{R},$$
where $a,\;b$  are fixed positive reals and proved generalized Ulam stability of these functional equations in $2$-normed spaces.

In 2018, Dung et al. \cite{dang} extended the fixed point theorem of Brzd\k{e}k et al. 
\cite[Theorem 1]{bcp} in metric spaces to $b$-metric spaces, in particular to quasi-Banach spaces, and studied the hyperstability  for the general  linear equation  in the setting of quasi-Banach spaces.
\vspace*{2mm}

A function $f : \mathbb{R} \to  X$ is a solution of the radical-type functional equation if and only if
\begin{align}\label{eq0}
f\left(\sqrt[3]{ax^{3}+by^{3}}\right)+f\left(\sqrt[3]{ax^{3}-by^{3}}\right)=cf(x)+df(y),\;\;x,y \in \mathbb{R},
\end{align} 
where $a, b, c$ and $d$ are nonzero real constants.

We  say  that a  function $f : \mathbb{R} \to  X$ fulfills  the radical-type functional equation \eqref{eq0} on $\mathbb{R}_0$ if and only if 
\begin{align}\label{eq000}
f\left(\sqrt[3]{ax^{3}+by^{3}}\right)+f\left(\sqrt[3]{ax^{3}-by^{3}}\right)=cf(x)+df(y),\;\; x,y \in \mathbb{R}_0,
\end{align}
where $\sqrt[3]{a}x\neq \pm \sqrt[3]{b}y$ and $a,b,c,d\in \mathbb{R}_0.$
\vspace*{2mm}

The contents of the paper are as follows:
\begin{itemize}
	\item[$\circ$] In Sect. \ref{some}, we introduce a new space called  \textit{quasi-$(2,\beta)$-Banach space} and we investigate also some results about this space.
	\item[$\circ$] In Sect. \ref{fpt}, we prove that fixed point theorem \cite[Theorem 2.1]{dang} remains  valid in the setting of quasi-$(2,\beta)$-Banach space and we derive from it many particular cases.
	\item[$\circ$] In Sect. \ref{sol}, we achieve the general solution of the functional equation \eqref{eq0}.
	\item[$\circ$] In the last  Sect. \ref{sol+}, we will apply our fixed point theorem to show   the hyperstability results for the radical-type functional equation  \eqref{eq0}, and  we finish this paper with  some consequences.
\end{itemize}

\vspace*{2mm}

It is well-known that the theory of 2-normed spaces  was initially  introduced  by G\"{a}hler \cite{sg,sg1} in the mid 1960s, and has been developed extensively in different subjects by others, for example \cite{el,fr,pa,w}.
Now, we recall by the definition  of a $(2, \beta)$-normed space and some preliminary results.

Let $0 < \beta\le  1$ be a fixed real number  and let  $X$ be a linear space over $\mathbb{K}$ with $dim\; X\ge 2$. A function  $\|\cdot,\cdot\|_{\beta}:X\times X\to \mathbb{R}_{+}$ is called a \textit{$(2,\beta)$-norm} on $X$ if and only if it satisfies: 
\begin{enumerate}
	\item [(N1)] $\|x,y\|_{\beta}= 0$ if and only if $x$ and $y$ are linearly dependent;
	\item [(N2)] $\|x,y\|_{\beta}=\|y,x\|_{\beta}$ for all $x, y \in X;$
	\item [(N3)] $\|\lambda x,y\|_{\beta}=|\lambda|^\beta \|x,y\|_{\beta}$ for all $x, y \in X$ and $\lambda\in \mathbb {K};$
	\item [(N4)]  $\|x , y+z\|_{\beta}\le \|x,y\|_{\beta}+\|x,z\|_{\beta}$ for all $x, y,z \in X$.
\end{enumerate}
 The pair $(X,\|\cdot,\cdot\|_{\beta})$ is called a $(2,\beta)$-\textit{normed space}. 

\begin{itemize}
	\item If $x \in X$ and $\|x, y\|_{\beta} = 0$ for all $y \in X$, then $x = 0$. Moreover; the functions $x \to \|x, y\|_{\beta}$ are continuous functions of $X$ into $\mathbb{R}_+$ for each fixed $y \in E$.
\item If $\beta=1$, then $(X,\|\cdot,\cdot\|_{\beta})$ becomes a linear $2$-normed space.
\end{itemize}

In 2006, Park \cite{prk} introduced the concept of quasi-2-normed spaces and quasi-$(2,p)$-normed spaces and studied the properties of these spaces. 
\begin{definition}\label{d2}
Let  $X$ be a linear space over $\mathbb{K}$ with $dim\; X\ge 2$ and $\|\cdot,\cdot\|_{q}:X\times X\to \mathbb{R}_{+}$ be 
a function such that 
\begin{enumerate}
	\item [(D1)] $\|x,y\|_{q}= 0$ if and only if $\{x,y\}$ is linearly dependent;
	\item [(D2)] $\|x,y\|_{q}=\|y,x\|_{q}$ for all $x, y \in X$;
	\item [(D3)] $\|\lambda x,y\|_{q}=|\lambda| \|x,y\|_{q}$ for all $x, y \in X$ and all $\lambda\in \mathbb {K}$;
	\item [(D4)] There is a constant $\kappa\geq 1$ such that $\|x , y+z\|_{q}\le \kappa(\|x,y\|_{q}+\|x,z\|_{q})$ for all $x, y,z \in X$.
\end{enumerate}
 Then $\|\cdot,\cdot\|_{q}$ is called a \textit{quasi-$2$-norm on} $X$, the smallest possible $\kappa$ is called \textit{the modulus of concavity} and $(X,\|\cdot,\cdot\|_{q}, \kappa)$ is called a \textit{quasi-$2$-normed space}.
\end{definition}
A quasi-2-norm $\|\cdot,\cdot\|_{q}$ is called a \textit{quasi-$p$-$2$-norm} $(0 < p \le 1)$ if
$$\|x+y,z\|_{q}^p\le \|x,z\|_{q}^p+\|y,z\|_{q}^p$$
for all $x, y,z \in X.$ The first difference between a quasi-$2$-norm and a $2$-norm is that the modulus of concavity of a quasi-norm is greater than or equal to 1, while that of a $2$-norm is equal to 1. This causes the quasi-$2$-norm to be not continuous in general, while a norm is always continuous. Moreover, by Aoki-Rolewicz Theorem \cite{mal}, each quasi-norm is equivalent to some p-norm. In \cite{prk}, Park has shown  the following theorem.
\begin{theorem}{\cite[Theorem 3]{prk}}\label{tp} Let $(X,\|\cdot,\cdot\|_{q}, \kappa)$ be a quasi-$2$-normed space. There is $p\in(0,1]$ and an equivalent quasi-$2$-norm $\||\cdot,\cdot|\|_{q}$ on $X$ satisfying
$$\||x+y,z|\|_{q}^p\le \||x,z|\|_{q}^p+\||y,z|\|_{q}^p$$
for all $x, y,z \in X,$ with 
$$\||x,z|\|_{q}:=\inf\Big\{\Big(\sum_{i=1}^n\|x_i,z\|_{q}^p\Big)^{1/p}: x=\sum_{i=1}^nx_i,\; x_i\in X, \;z\in X,\; n\in\mathbb{N}\Big\}$$
and $p=\log_{2\kappa}2.$
\end{theorem}
\vspace*{2mm}

\section{Some properties of quasi-$(2,\beta)$-Banach spaces}\label{some}
In this section, we generalize the concept of quasi-$2$-normed spaces and we prove some properties of the quasi-$(2,\beta)$-Banach spaces.

\begin{definition}\label{d3}
Let $0 < \beta\le  1$ be a fixed real number  and  $X$ be a linear space over $\mathbb{K}$ with $dim\; X\ge 2$ and let $\|\cdot,\cdot\|_{q,\beta}:X\times X\to \mathbb{R}_{+}$ be 
a function such that 
\begin{enumerate}
	\item [(B1)] $\|x,y\|_{q,\beta}= 0$ if and only if $\{x,y\}$ is linearly dependent;
	\item [(B2)] $\|x,y\|_{q,\beta}=\|y,x\|_{q,\beta}$ for all $x, y \in X$;
	\item [(B3)] $\|\lambda x,y\|_{q,\beta}=|\lambda|^{\beta} \|x,y\|_{q,\beta}$ for all $x, y \in X$ and all $\lambda\in \mathbb {K}$;
	\item [(B4)] There is a constant $\kappa\geq 1$ such that $\|x , y+z\|_{q,\beta}\le \kappa(\|x,y\|_{q,\beta}+\|x,z\|_{q,\beta})$ for all $x, y,z \in X$.
\end{enumerate}
 Then $\|\cdot,\cdot\|_{q,\beta}$ is called a \textit{quasi-$(2,\beta)$-norm} on $X$, the smallest possible $\kappa$ is called the \textit{modulus of concavity} and $(X,\|\cdot,\cdot\|_{q,\beta}, \kappa)$ is called a \textit{quasi-$(2,\beta)$-normed space}.
\end{definition}
A quasi-2-norm $\|\cdot,\cdot\|_{q,\beta}$ is called a \textit{quasi-$p$-$(2,\beta)$-norm} $(0 < p \le 1)$ if
$$\|x+y,z\|_{q,\beta}^p\le \|x,z\|_{q,\beta}^p+\|y,z\|_{q,\beta}^p,\;\;\;x, y,z \in X,$$
and we have
$$\big|\|x,z\|_{q,\beta}^p-\|y,z\|_{q,\beta}^p\big|\le \|x+y,z\|_{q,\beta}^p,\;\;\;x, y,z \in X.$$
\begin{example} Let $X$ be a linear space with $dim X\ge 2$, and let $\|\cdot,\cdot\|_{\beta}$ be a $(2,\beta)$-norm on $X.$ Then
$$\|x,y\|_{q,\beta}=C\|x,y\|_{\beta},\;\;x,y\in X\;\; (\text{with}\;\; C>0 \;\;\text{is a constant}),$$
is a quasi-$(2,\beta)$-norm on $X$, and $(X,\|\cdot,\cdot\|_{q,\beta})$ is  a quasi-$(2,\beta)$-normed space
\end{example}
\begin{example}\label{e1} If $X$ is a quasi-2-norm space with the quasi-2-norm $\|\cdot,\cdot\|_{q}$ and the modulus of concavity $\kappa\geq 1$, then it is a quasi-$(2,\beta)$-normed space with the quasi-$(2,\beta)$-norm $\|x,y\|_{q,\beta}=\|x,y\|_{q}^{\beta}$ for all $x,y\in X$
and $0 < \beta\le  1$ is a fixed real number.
\end{example}
\begin{proof}
Indeed, for every $x,y,z\in X$ and $\lambda\in \mathbb{K}$, we have
$$\;\;\;\|x,y\|_{q,\beta}=0\Leftrightarrow \|x,y\|_{q}=0\Leftrightarrow \; \{x,y\}\;\;\text{ is linearly dependent},$$
$$\|y,x\|_{q,\beta}= \|y,x\|_{q}^{\beta}=\|x,y\|_{q}^{\beta}=\|x,y\|_{q,\beta}$$
$$\|\lambda x,y\|_{q,\beta} =\|\lambda x,y\|_{q}^{\beta}=|\lambda|^{\beta}\|x,y\|_{q,\beta}$$
and
\begin{align*}
\|x+y,z\|_{q,\beta}&=\|x+y,z\|_{q}^{\beta}\le  \kappa^{\beta}(\|x,z\|_{q}+\|y,z\|_{q})^{\beta}\\&\le \kappa^{\beta}(\|x,z\|_{q}^{\beta}+\|y,z\|_{q}^{\beta}) \\&= \kappa^{\beta}(\|x,z\|_{q,\beta}+\|y,z\|_{q,\beta}).
\end{align*}
As $\kappa^{\beta}\ge 1,$ then $(X,\|\cdot,\cdot\|_{q,\beta},\kappa^{\beta})$ is a quasi-$(2,\beta)$-normed space.
\end{proof}
\begin{lemma}\label{l1}
Let $(X,\|\cdot,\cdot\|_{q,\beta}, \kappa)$ be a quasi-$(2,\beta)$-normed space  with $0 < \beta\le  1$. If $x \in X$ and $\|x,y\|_{q,\beta}=0$ for all $y \in X$, then $x = 0.$
\end{lemma}
\begin{proof} Suppose that $x\neq 0$. Since $dim X\ge 2$, choose $y\in X$ such that $\{x,y\}$ is linearly independent, and  by  Definition \ref{d3} (B1) we have $\|x,y\|_{q,\beta}\neq 0$. This is a contradiction, thus $x=0.$
\end{proof}
\begin{definition}\label{d4}
Let $X$ be a quasi-$(2,\beta)$-normed space  with $0 < \beta\le  1$.
\begin{enumerate}
	\item [(1)] A sequence $\{x_n\}$ in $ X$ is called a \textit{Cauchy sequence} if there are two points $y, z \in X$
such that $y$ and $z$ are linearly independent,
$$\lim_{m,n\to\infty}\|x_m-x_n,y\|_{q,\beta}=0=\lim_{m,n\to\infty}\|x_m-x_n,z\|_{q,\beta}.$$
	\item [(2)] A sequence $\{x_n\}$ in $ X$ is called a \textit{convergent sequence} if there is an $x \in X$ such that
	$$\lim_{n\to\infty}\|x_n-x,y\|_{q,\beta}=0\;\; \text{for all}\;\; y\in X.$$
	In this case we write we also write $\lim_{n\to\infty}x_n=x.$
	\item [(3)] A  quasi-$(2,\beta)$-normed space in which every Cauchy sequence is a convergent sequence is called a \textit{quasi-$(2,\beta)$-Banach space}.
\end{enumerate}
\end{definition}
\begin{remark}
The functions $x\to \|x,y\|_{q,\beta}$ are not necessary  continuous functions of $X$ into $\mathbb{R}$ for each fixed $y \in X.$
\end{remark}

\begin{theorem}\label{tpp} Let $(X,\|\cdot,\cdot\|_{q,\beta}, \kappa)$ be a quasi-$(2,\beta)$-normed space, with $0 < \beta\le  1$. There is $p\in(0,1]$ and an equivalent quasi-2-norm $\||\cdot,\cdot|\|_{q,\beta}$ on $X$ satisfying
$$\||x+y,z|\|_{q,\beta}^p\le \||x,z|\|_{q,\beta}^p+\||y,z|\|_{q,\beta}^p$$
for all $x, y,z \in X,$ with 
$$\||x,z|\|_{q,\beta}:=\inf\Big\{\Big(\sum_{i=1}^n\|x_i,z\|_{q,\beta}^{p/\beta}\Big)^{\beta/p}:\; x=\sum_{i=1}^nx_i,\; x_i\in X,\; z\in X,\; n\in\mathbb{N}\Big\}$$
and $p=\beta\log_{2\kappa}2.$
\end{theorem}
\begin{proof}Let $(X,\|\cdot,\cdot\|_{q,\beta}, \kappa)$ be a quasi-$(2,\beta)$-normed space with $0 < \beta\le  1,$ and let $x,z\in X$ such that 
$$\||x,z|\|_{q}:=\inf\Big\{\Big(\sum_{i=1}^n\|x_i,z\|_{q,\beta}^{p/\beta}\Big)^{1/p}: x=\sum_{i=1}^nx_i, x_i\in X,  n\in\mathbb{N}\Big\},$$
$p=\beta\log_{2\kappa}2.$ 

We will demonstrate  that $\|\cdot,\cdot\|_{q}:=\|\cdot,\cdot\|_{q,\beta}^{\frac{1}{\beta}}$ is a quasi-2-norm on $X$ with the modulus of concavity $\frac{(2\kappa)^{\frac{1}{\beta}}}{2}.$ 

For that, we assume that $\|\cdot,\cdot\|_{q,\beta}$ is a quasi-$(2,\beta)$-norm on $X$. Then for every $x,y,z\in X$ and for every $\lambda\in \mathbb{K}$, we have
 $$\|x,y\|_{q}=0\Leftrightarrow\|x,y\|_{q,\beta}^{\frac{1}{\beta}}=0\Leftrightarrow\|x,y\|_{q,\beta}=0\Leftrightarrow \{x,y\} \;\;\;\text{is linearly dependent};$$
 $$\|\lambda x,y\|_{q}=\|\lambda x,y\|_{q,\beta}^{\frac{1}{\beta}}=|\lambda|\| x,y\|_{q},\quad \;\;\;\|y,x\|_{q}=\| x,y\|_{q};$$
and 
\begin{align}\label{inq}
\|x+y,z\|_{q}=\|x+y,z\|_{q,\beta}^{1/\beta}\le  \kappa^{1/\beta}(\|x,z\|_{q,\beta}+\|y,z\|_{q,\beta})^{1/\beta}.
\end{align}
 Let $g(t)=t^{r}$ for all $t>0$ with $r\ge 1$. It is easy to check that $g$ is a convex function,  hence  by the definition of convexity, we have
$$g\left(\frac{t+s}{2}\right)\le \frac{1}{2}(g(t)+g(s))\;\;\text{ for all}\;\; t,s>0,$$
i.e., $(t+s)^r\le 2^{r-1}(t^r+s^r).$ So, \eqref{inq} becomes
\begin{align*}
\|x+y,z\|_{q}\le  \frac{(2\kappa)^{1/\beta}}{2}\big(\|x,z\|_{q,\beta}^{1/\beta}+\|y,z\|_{q,\beta}^{1/\beta}\big)= \frac{(2\kappa)^{1/\beta}}{2}\big(\|x,z\|_{q}+\|y,z\|_{q}\big).
\end{align*}
It follows that $\Big(X,\|\cdot,\cdot\|_{q}, \frac{(2\kappa)^{1/\beta}}{2}\Big)$ is a quasi-2-normed space. From Theorem \ref{tp}, we get a quasi-2-norm $\||\cdot,\cdot|\|_{q}$  on $X$ satisfying
$$\||x+y,z|\|_{q}^p\le \||x,z|\|_{q}^p+\||y,z|\|_{q}^p,\;\;\; x,y,z\in X,$$
and there exist $\mu_1, \mu_2>0$ such that
$$\mu_1\|x,z\|_{q,\beta}^{1/\beta}=\mu_1\|x,z\|_{q}\le \||x,z|\|_{q}\le \mu_2\|x,z\|_{q}=\mu_2\|x,z\|_{q,\beta}^{1/\beta}\;\;\;\;\; \forall x,y\in X.$$ 
Since $0 < \beta\le  1$ and by Example \ref{e1}, we get $\||\cdot,\cdot|\|_{q,\beta}:= \||\cdot,\cdot|\|_{q}^{\beta}$ is also a 
quasi-$(2,\beta)$-norm on $X$ and
\begin{align*}
\||x+y,z|\|_{q,\beta}^p&=\||x+y,z|\|_{q}^{p\beta}\le (\||x,z|\|_{q}^{p}+\||y,z|\|_{q}^{p})^{\beta}\\&\le \||x,z|\|_{q}^{p\beta}+\||y,z|\|_{q}^{p\beta}=\||x,z|\|_{q,\beta}^p+\||y,z|\|_{q,\beta}^p
\end{align*}
for all $x,y,z\in X.$ Also we have
$$\mu_1^{\beta}\|x,z\|_{q,\beta}\le \||x,z|\|_{q}^{\beta}\le \mu_2^{\beta}\|x,z\|_{q,\beta}\;\;\; \quad\forall x,z\in X,$$ 
i.e.,
\begin{align}\label{eq}
C_1\|x,z\|_{q,\beta}\le \||x,z|\|_{q,\beta}\le C_2\|x,z\|_{q,\beta}\;\;\;\quad \forall x,z\in X,
\end{align}
with $C_i=\mu_i^{\beta}$ for $i\in\{1,2\}$. This completes the proof.
\end{proof}

\begin{remark} It follows from Theorem \ref{tpp} that
\begin{enumerate}
\item[(i)] $\big|\||x,z|\|_{q,\beta}^p-\||y,z|\|_{q,\beta}^p\big|\le \||x-y,z|\|_{q,\beta}^p,\;\;x,y,z\in X.$
	\item[(ii)] The functions $x\to \||x,y|\|_{q,\beta}$ are   continuous functions of $X$ into $\mathbb{R}$ for each fixed $y \in X.$
	\item[(iii)] If $\beta=1$, we get the Theorem \ref{tp}.
\end{enumerate}
\end{remark}

\section{A new fixed point theorem}\label{fpt}
In this section, we prove that the fixed point theorem \cite[Theorem 2.1]{dang} remains valid in the setting of quasi-$(2,\beta)$-Banach space. Let us introduce the following four hypotheses:\\

\textbf{(A1)} $W$ is a nonempty set, $(X,\|\cdot,\cdot\|_{q,\beta},\kappa)$ is a quasi-$(2,\beta)$-Banach space. 
\\

 \textbf{(A2)} $f_{i}:W\to W$ and $L_{i}:W\times X\to \mathbb{R}_+$ are given maps for $i=1,\ldots,j$.
\\

\textbf{(A3)} $\mathcal{T}:X^W\to X^W$ is an operator satisfying the inequality
$$\left\|(\mathcal{T}\xi)(x)-(\mathcal{T}\mu)(x),y\right\|_{q,\beta}\leq \sum_{i=1}^{j}L_{i}(x,y)\left\|\xi(f_{i}(x))-\mu(f_{i}(x)),y\right\|_{q,\beta}$$
for all $\xi, \mu\in X^W$  and $ (x,y)\in W\times X$.
\\

\textbf{(A4)} $\Lambda:\mathbb{R}_{+}^{W\times X}\to \mathbb{R}_{+}^{W\times X}$ is a linear operator defined by 
$$(\Lambda\delta)(x,y):=\sum_{i=1}^{j}L_{i}(x,y)\delta(f_{i}(x),y),\quad\delta\in \mathbb{R}_{+}^{W\times X},\;\;(x,y)\in W\times X. $$

\begin{theorem} \label{pl1} Let hypotheses {\rm\textbf{(A1)}-\textbf{(A4)}} be valid, and let $\varepsilon:W\times X\to\mathbb{R}_{+},$  $\varphi: W\to X$ satisfy the conditions
\begin{align}\label{f1}
\left\|(\mathcal{T}\varphi)(x)-\varphi(x),y\right\|_{\beta}\leq \varepsilon(x,y),\;\;(x,y)\in W\times X,
\end{align}
\begin{align}\label{f2}
\varepsilon^{*}(x,y):=\sum_{n=0}^{\infty}(\Lambda^n\varepsilon)^{\theta}(x,y)< \infty, \;\; (x,y)\in W\times X,
\end{align}
where $\theta=\beta\log_{2\kappa}2.$ Then the  limit
\begin{align}\label{f3}
\psi(x)=\lim_{n\to \infty}\mathcal{T}^n\varphi(x), \;\; x\in W
\end{align}
exists and the  function $\psi:W\to X$  defined by \eqref{f3} is a fixed point of $\mathcal{T}$ satisfying
\begin{align}\label{f4}
\left\|\varphi(x)-\psi(x),y\right\|^{\theta}_{q,\beta} \leq K\varepsilon^{*}(x,y), \;\;(x,y)\in W\times X.
\end{align}
for some constant $K>0$. Moreover, if
 \begin{align}\label{f04}
\varepsilon^{*}(x,y)\le \left(M\sum_{n=0}^{\infty}(\Lambda^n\varepsilon)(x,y)\right)^{\theta}< \infty
\end{align}
for some positive real number $M$, then $\psi$ is the unique fixed point of $\mathcal{T}$ satisfying \eqref{f4}.
\end{theorem} 
\begin{proof} It is easy to show by induction that, for any $n\in\mathbb{N}_0$
\begin{align}\label{f5}
\big\|(\mathcal{T}^{n+1}\varphi)(x)-(\mathcal{T}^{n}\varphi)(x),y\big\|_{q,\beta}\leq (\Lambda^n\varepsilon)(x,y),\;\; (x,y)\in W\times X.
\end{align}

Indeed, by \eqref{f1}, we have 
$$\big\|\varphi(x)-(\mathcal{T}\varphi)(x),y\big\|_{q,\beta}\leq \varepsilon(x,y)=(\Lambda^0\varepsilon)(x,y),\;\; (x,y)\in W\times X.$$
Then the case $n=0$ is true. Now, fix an $n\in\mathbb{N}_0$ and suppose that \eqref{f5} is true. Then, by (A3)-(A4), for any $(x,y)\in W\times X$, we get
 \begin{align*}
\big\|(\mathcal{T}^{n+1}\varphi)(x)&-(\mathcal{T}^{n+2}\varphi)(x),x_2,y\big\|_{q,\beta}=
\big\|\mathcal{T}(\mathcal{T}^{m+1}\varphi)(x)-\mathcal{T}(\mathcal{T}^{m}\varphi)(x),y\big\|_{q,\beta}
\\&\leq \sum_{1\le i\le j}L_i(x,y)\big\|\mathcal{T}^{n}\varphi(f_i(x))-\mathcal{T}^{n+1}\varphi(f_i(x)),y\big\|_{q,\beta}
\\&\leq \sum_{1\le i\le j}L_i(x,y)(\Lambda^n\varepsilon)(f_i(x),y)
=(\Lambda^{n+1}\varepsilon)(x,y).
\end{align*}
So, \eqref{f5} holds for all $n\in\mathbb{N}_0$.

Next, from \eqref{f5}, \eqref{eq} and Theorem \ref{tpp}, for every $k\in\mathbb{N},$ $n\in\mathbb{N}_0$ and $(x,y)\in W\times X$, we have 
\begin{align}\label{f6}
\big\|\big|(\mathcal{T}^{n}\varphi)(x)-(\mathcal{T}^{n+k}\varphi)(x),&y\big|\big\|_{q,\beta}^{\theta}=
\big\|\big|\sum_{i=0}^{k-1}\big((\mathcal{T}^{n+i}\varphi)(x)-(\mathcal{T}^{n+i+1}\varphi)(x)\big),y\big|\big\|_{q,\beta}^{\theta}\nonumber
\\&\leq \sum_{i=0}^{k-1}\big\|\big|(\mathcal{T}^{n+i}\varphi)(x)-(\mathcal{T}^{n+i+1}\varphi)(x),y\big|\big\|_{q,\beta}^{\theta}\nonumber
\\&\le C_2^{\theta} \sum_{i=0}^{k-1}\big\|(\mathcal{T}^{n+i}\varphi)(x)-(\mathcal{T}^{n+i+1}\varphi)(x),y\big\|_{q,\beta}^{\theta}\nonumber
\\&\le C_2^{\theta} \sum_{i=0}^{k-1}(\Lambda^{n+i}\varepsilon)^{\theta}(x,y)\nonumber
= C_2^{\theta} \sum_{i=n}^{n+k-1}(\Lambda^{i}\varepsilon)^{\theta}(x,y)\nonumber
\\&\le C_2^{\theta}\varepsilon^{*}(x,y)\;\; \text{for some}\;\; C_2>0.
\end{align}
By the convergence of the series $\sum_{n\ge 0}(\Lambda^{n}\varepsilon)^{\theta}(x,y)$, it follows from \eqref{f6} that, for every $x\in W$, $\big\{(\mathcal{T}^{n}\varphi)(x)\big\}_{n\in\mathbb{N}}$ is a Cauchy sequence in $(X,\||\cdot,\cdot|\|_{q,\beta}).$ By Theorem \ref{tpp}, $\big\{(\mathcal{T}^{n}\varphi)(x)\big\}_{n\in\mathbb{N}}$ is also a Cauchy sequence in $(X,\|\cdot,\cdot\|_{q,\beta},\kappa)$. As $(X,\|\cdot,\cdot\|_{q,\beta},\kappa)$ is a quasi-$(2,\beta)$-Banach space, then the limit $\psi(x):=\lim_{m\to \infty}(\mathcal{T}^m\varphi)(x)$ exists for any $x\in W$, so \eqref{f3} holds.

Since $\||\cdot,\cdot|\|_{q,\beta}$ is continuous and taking $n=0$ and $k\to\infty$ in \eqref{f6}, we get
\begin{align}\label{f7}
\||\varphi(x)-\psi(x),y|\|_{q,\beta}^{\theta}\le  C_2^{\theta}\varepsilon^{*}(x,y)
\end{align}
for all $(x,y)\in W\times X$  and for some $C_2>0.$ From \eqref{eq}, we find
$$
\|\varphi(x)-\psi(x),y\|_{q,\beta}^{\theta}\le C_1^{-\theta}\||\varphi(x)-\psi(x),y|\|_{q,\beta}^{\theta}\le  (C_2/C_1)^{\theta}\varepsilon^{*}(x,y)
$$
for all $(x,y)\in W\times X$  and for some $C_1,C_2>0,$ so \eqref{f4} holds with $K:=(C_2/C_1)^{\theta}$.

By applying (A3), \eqref{f3} and \eqref{eq}, we get 
\begin{align}\label{f8}
\big\|\big(\mathcal{T}^{n+1}\varphi)(x)-(\mathcal{T}\psi)(x),&y\big\|_{q,\beta}
\leq \sum_{i=1}^{j}L_i(x,y)\big\|\big(\mathcal{T}^{n}\varphi)(f_i(x))-\psi(f_i(x)),y\big\|_{q,\beta}
\nonumber
\\& \le C_1^{-1}  \sum_{i=1}^{j}L_i(x,y)\big\|\big|\big(\mathcal{T}^{n}\varphi)(f_i(x))-\psi(f_i(x)),y\big|\big\|_{q,\beta}
\end{align}
for all $(x,y)\in W\times X$,  $n\in\mathbb{N}$  and for some $C_1>0.$ Letting $n\to\infty$ in \eqref{f8}, we find $$\lim_{n\to\infty}\big\|\big(\mathcal{T}^{n+1}\varphi)(x)-(\mathcal{T}\psi)(x),y\big\|_{q,\beta}=0$$ for all $(x,y)\in W\times X$, that is, 
$\lim_{n\to\infty}\big(\mathcal{T}^{n+1}\varphi)(x)=(\mathcal{T}\psi)(x)$ for all $x\in W$. This prove $\mathcal{T}\psi=\psi.$ So $\psi$ is a fixed point of $\mathcal{T}$ that satisfies \eqref{f4}.

 It remains to prove the uniqueness  of $\psi$. Let $\gamma$ be also a fixed point of $\mathcal{T}$ satisfying \eqref{f4}. For every $m\in\mathbb{N}_0,$ we show that
\begin{align}\label{f9}
\|\psi(x)-\gamma(x),y\|_{q,\beta}
=\big\|\big(\mathcal{T}^{m}\psi)(x)-(\mathcal{T}^m\gamma)(x),y\big\|_{q,\beta}
\leq (2K)^{1/\theta}M\sum_{i=m}^{\infty} (\Lambda^{i}\varepsilon)(x,y)
\end{align}
for all $(x,y)\in W\times X$  and for some $K,M>0.$ Indeed, for $m=0$ and from \eqref{f7}, we have
$$\||\psi(x)-\gamma(x),y|\|_{q,\beta}^{\theta}\le\||\psi(x)-\varphi(x),y|\|_{q,\beta}^{\theta}+\||\varphi(x)-\gamma(x),y|\|_{q,\beta}^{\theta}\le  2C_2^{\theta}\varepsilon^{*}(x,y)$$
for all $(x,y)\in W\times X$  and for some $C_2>0.$ From \eqref{eq} and \eqref{f04} we have
\begin{align*}
\|\psi(x)-\gamma(x),y\|_{q,\beta}^{\theta}&\le C_1^{-\theta}\||\psi(x)-\gamma(x),y|\|_{q,\beta}^{\theta}
\le (C_2/C_1)^{\theta}\varepsilon^{*}(x,y)
\\&\le 2K \left(M\sum_{i=0}^{\infty} (\Lambda^{i}\varepsilon)(x,y)\right)^{\theta},\;\;(x,y)\in W\times X.
\end{align*}
Thus, \eqref{f9} holds for $m=0.$

 Now, assume that \eqref{f9} is valid for some $m\in\mathbb{N}_0.$ By (A3) and \eqref{f9}, we obtain 
\begin{align*}
\big\|\big(\mathcal{T}^{m+1}\psi)(x)-(\mathcal{T}^{m+1}\gamma)(x),y\big\|_{q,\beta}
&\le \sum_{k=1}^jL_k(x,y)\big\|\big(\mathcal{T}^{m}\psi)(f_k(x))-(\mathcal{T}^{m}\gamma)(f_k(x)),y\big\|_{q,\beta}
\\&\le\sum_{k=1}^jL_k(x,y) (2K)^{1/\theta}M\sum_{i=m}^{\infty} (\Lambda^{i}\varepsilon)(x,y)
\\&= (2K)^{1/\theta}M\sum_{i=m+1}^{\infty} (\Lambda^{i}\varepsilon)(x,y),\;\; x,y\in X.
\end{align*}
So, \eqref{f9} holds for all $m\in\mathbb{N}_0.$ It follows from \eqref{eq} and \eqref{f9} that 
\begin{align}\label{f11}
\big\|\big|\big(\mathcal{T}^{m}\psi)(x)-(\mathcal{T}^{m}\gamma)(x),y\big|\big\|_{q,\beta}^{\theta}
&\le C_2^{\theta}\big\|\big(\mathcal{T}^{m}\psi)(x)-(\mathcal{T}^{m}\gamma)(x),y\big\|_{q,\beta}^{\theta}\nonumber
\\&\le 2KC_2^{\theta}  \left(M\sum_{i=m}^{\infty} (\Lambda^{i}\varepsilon)(x,y)\right)^{\theta}
\end{align}
for all $(x,y)\in W\times X$ and all $m\in\mathbb{N}_0.$ Letting $m\to\infty$ in \eqref{f11} and from \eqref{f04}, we get $\big\|\big|\psi(x)-\gamma(x),y\big|\big\|_{q,\beta}=0$ for all $y\in X.$ That is $\psi\equiv\gamma$. This completes the proof.
\end{proof}
\begin{corol}\label{cl1}Assume that hypotheses {\rm\textbf{(A1)}-\textbf{(A4)}} are satisfied. Suppose that there exist two functions $\varepsilon:W\times X\to\mathbb{R}_{+}$  and $\varphi: W\to X$ such that \eqref{f1} holds and
\begin{align}\label{f13}
(\Lambda\varepsilon)(x,y)=q\varepsilon(x,y), \;\;(x,y)\in W\times X,\;\;\;\;q\in [0,1).
\end{align}
Then the limit \eqref{f3} exists and the function $\psi:W\to X$ so defined is the unique fixed point of $\mathcal{T}$ with 
$$\left\|\varphi(x)-\psi(x),y\right\|\le \frac{K\varepsilon^{\theta}(x,y)}{1-q^{\theta}}, \;\;\;\;(x,y)\in W\times X,\;\;\;\;q\in [0,1)$$
for some $K>0,$ where $\theta=\beta\log_{2\kappa}2.$
\end{corol}

\begin{proof}It follows from \eqref{f13} that
$$(\Lambda^n\varepsilon)(x,y)=q^n\varepsilon(x,y), \;\;(x,y)\in W\times X,\;n\in \mathbb{N}_0.$$
Therefore, for every $(x,y)\in W\times X,$

$$\varepsilon^{*}(x,y)=\sum_{n=0}^{\infty} (\Lambda^{n}\varepsilon)^{\theta}(x,y)=\sum_{n=0}^{\infty} q^{n\theta}(\varepsilon(x,y))^{\theta}=\frac{\varepsilon^{\theta}(x,y)}{1-q^{\theta}}<\infty$$
where $\theta=\beta\log_{2\kappa}2.$ So, condition \eqref{f2} holds. Moreover
$$\left(\sum_{n=0}^{\infty} (\Lambda^{n}\varepsilon)(x,y)\right)^{\theta}=\left(\sum_{n=0}^{\infty} q^{n}\varepsilon(x,y)\right)^{\theta}=\frac{\varepsilon^{\theta}(x,y)}{(1-q)^{\theta}}<\infty$$
for all $(x,y)\in W\times X.$ As 
$$\frac{\varepsilon^{\theta}(x,y)}{1-q^{\theta}}\le\frac{\varepsilon^{\theta}(x,y)}{(1-q)^{\theta}},\quad\;\;(x,y)\in W\times X,$$
then condition \eqref{f04} holds and our assertion follows from Theorem \ref{pl1} and its proof.
\end{proof}
\begin{remark} (i) If $(X,\|\cdot,\cdot\|,\kappa)$ is a quasi-$2$-Banach in Theorem \ref{pl1}, then $\theta=\log_{2\kappa}2$ (i.e., $\beta=1$) and Theorem \ref{pl1}  remains true. 

	(ii) If $(X,\|\cdot,\cdot\|_{\beta})$ is a $(2,\beta)$-Banach in Theorem \ref{pl1}, then $\kappa=1$ and Theorem \ref{pl1}  remains true with $\||\cdot,\cdot|\|_{\beta}=\|\cdot,\cdot\|_{\beta}$.
	
	(iii) If $(X,\|\cdot,\cdot\|)$ is a $2$-Banach in Theorem \ref{pl1}, then $\kappa=1=\beta$ we obtain \cite[Theorem 1]{krs} with $\||\cdot,\cdot|\|=\|\cdot,\cdot\|$. Moreover, we have
	$$\left\|\varphi(x)-\psi(x),y\right\| \leq \varepsilon^{*}(x,y).
$$	
\end{remark}

\section{General Solution of Eq.~\eqref{eq0}}\label{sol}

In this section, we give  the general solution of the radical-type functional equation \eqref{eq0} by using some results that are reported in  \cite{ATNAA}.

\begin{prop} \label{p1}
Let $a,b,c,d\in\mathbb{R}\setminus\{0\}$ be fixed numbers and $\mathcal{V}$ a real vector space. A function $f : \mathbb{R} \to  \mathcal{V} $ satisfies  \eqref{eq0}  if and only if there exists a solution $g : \mathbb{R} \to  \mathcal{V}$ of the equation
\begin{align}\label{fssi}
g(ax +by)+g(ax-by) = cg(x) + dg(y),\quad x,y \in\mathbb{R},
\end{align}
such that  $f(x)=g\big(x^3\big)$ for all $x\in\mathbb{R}.$
\end{prop}
\begin{proof}  It is clear  that if $f : \mathbb{R}\to  \mathcal{V} $ has form
$f(x)=g\big(x^3\big)$ for  $x\in\mathbb{R},$  with $g$ satisfies  \eqref{fssi}, then it is a solution of \eqref{eq0}.

On the other hand, if $f : \mathbb{R} \to  \mathcal{V} $ is a solution to \eqref{eq0}, and setting $h(x)=f\big(\sqrt[3]{x}\big)$ for all $x\in \mathbb{R}$, then from \eqref{eq0}, we get 
\begin{align*}
h\left(ax^3+by^3\right)&+h\left(ax^3-by^3\right)=f\left(\sqrt[3]{ax^3+by^3}\right)+f\left(\sqrt[3]{ax^-by^3}\right)\\
&=cf(x)+df(y)=ch\big(x^3\big)+dh\big(y^3\big),\quad x,y\in \mathbb{R},
\end{align*}
where $a,b,c,d\in\mathbb{R}\setminus\{0\}$. So, $h$ satisfies \eqref{fssi}.    Thus it suffices to take $g\equiv h$. 
\end{proof}

\begin{lemma} \label{lem1}
Let $\mathcal{V}$ be a real vector space and  $f : \mathbb{R} \to  \mathcal{V}$ be a function satisfying  \eqref{eq0} with $a,b,c,d\in\mathbb{R}\setminus\{0\}$ and $c+d\neq 2$. Then
\begin{itemize}
	\item [(i)] $f$ satisfies the functional equation 
	\begin{equation}\label{tic}
	f\left(\sqrt[3]{x^{3}+y^{3}}\right)+f\left(\sqrt[3]{x^{3}-y^{3}}\right)=2f(x)+2f(y),\,\quad x,y\in\mathbb{R}.
	\end{equation}
	\item [(ii)]$f$ is a sextic mapping and if $f$ is continuous, then $f(x)=x^6f(1)$ for all $x\in\mathbb{R}_+$.
\end{itemize}
\end{lemma}
\begin{proof} (i) Suppose  that $f$ satisfies \eqref{eq0}, then by letting  $x= y = 0$ in \eqref{eq0}, we have  $f (0) = 0$ because $c+d \neq 2.$ Replacing $y$ by $-y$ in \eqref{eq0}, we get 
\begin{equation}\label{tii}
	f\left(\sqrt[3]{ax^{3}-by^{3}}\right)+f\left(\sqrt[3]{ax^{3}+by^{3}}\right)=cf(x)+df(-y),\,\quad x,y\in\mathbb{R}.
	\end{equation}
If we compare \eqref{eq0} with \eqref{tii}, we obtain that the function $f$ is even. 

Taking first $y=0$ and next $x=0$ in \eqref{eq0}, we get
$$f\big(\sqrt[3]{a}x\big)=\frac{c}{2}f(x),\quad\;\;\;\;f\big(\sqrt[3]{b}y\big)=\frac{d}{2}f(y).$$
  So, 
\begin{equation}\label{rc}
f\big(\sqrt[3]{ab}x\big)=\frac{cd}{4}f(x),\;\;x\in \mathbb{R}.
\end{equation}
 Replacing $(x,y)$ by $\big(\sqrt[3]{b}x,\sqrt[3]{a}y\big)$ in \eqref{eq0}, we obtain
\begin{equation}\label{eqgrc1}
f \left(\sqrt[3]{abx^3 + aby^3}\right)+f \left(\sqrt[3]{abx^3 - aby^3}\right)= cf \big(\sqrt[3]{b}x\big) + df \big(\sqrt[3]{a}y\big)\;\;x\in \mathbb{R}.
\end{equation}
 It follows from \eqref{rc} and \eqref{eqgrc1} that $f$ is a solution of \eqref{tic}.

(ii) As $f$  satisfies \eqref{tic}  and by \cite[Remark 1]{fass}, we get the desired results.
\end{proof}
\begin{theorem}\label{tt2} Let $a,b,c,d\in\mathbb{R}\setminus\{0\}$ be fixed numbers and $\mathcal{ V}$  a real vector space. A function $f : \mathbb{R} \to  \mathcal{V} $ satisfies  \eqref{eq0}  if and only if:
\begin{itemize}
	\item [(i)] In the case $c+d\neq 2,$  there exists a quadratic mapping $Q : \mathbb{R} \to  \mathcal{V}$  such that $f(x)=Q\big(x^3\big)$ for all $ x\in\mathbb{R}$ and
	\begin{equation}\label{tic+c}
	Q(ax)=\frac{c}{2}Q(x),\qquad Q(bx)=\frac{d}{2}Q(x),\qquad x\in\mathbb{R}.
	\end{equation}
	
	\item [(ii)]In the case $c+d= 2,$  there exist $w \in\mathcal{V}$ and a quadratic  mapping $Q : \mathbb{R} \to  \mathcal{V},$  such that \eqref{tic+c} holds and $f(x)=Q\big(x^3\big)+w$ for all $x\in\mathbb{R}.$
\end{itemize}
\end{theorem}

\begin{proof} (i) Assume that $f$ satisfies \eqref{eq0} and $c+d\neq 2$. Then by Lemma \ref{lem1}, we obtain that $f$ satisfies \eqref{tic}. Applying \cite[Theorem 2.1]{fass},  there exists a quadratic mapping $Q : \mathbb{R} \to  \mathcal{V}$ such that $f(x)=Q\big(x^3\big)$ for all $x\in\mathbb{R}$. Also, we have $f(0)=0$ and
$$f\big(\sqrt[3]{a}x\big)=\frac{c}{2}f(x),\;\;\;\;f\big(\sqrt[3]{b}x\big)=\frac{d}{2}f(x),\;\;x\in\mathbb{R}.$$
So,
$$Q(ax)=\frac{c}{2}Q(x)\;\text{and}\;\;Q(bx)=\frac{d}{2}Q(x),\quad x\in\mathbb{R}. $$

(ii) Suppose that $f$ satisfies \eqref{eq0} and $c+d= 2$. Putting
$$f_0(x):=f(x)-f(0),\quad x\in\mathbb{R}.$$
Then $f_0(0)=0$ and $f_0$ satisfies \eqref{eq0}. Similarly to the proof of Lemma \ref{lem1}, we get 
$$f_0\big(\sqrt[3]{a}x\big)=\frac{c}{2}f_0(x),\;\;\;\;f_0\big(\sqrt[3]{b}x\big)=\frac{d}{2}f_0(x),\;\;x\in\mathbb{R},$$
and $f_0$ satisfies \eqref{tic}, so by \cite[Theorem 2.1]{fass}, there exists a quadratic mapping $Q : \mathbb{R} \to  \mathcal{V}$ such that $f_0(x)=Q\big(x^3\big)$ and $Q$ satisfies \eqref{tic+c}. Hence $f(x)=Q\big(x^3\big)+w,$ with $w:=f(0).$

The converse is easy to check. This completes the proof.
\end{proof}
\begin{remark}\label{rq1} It is well known (see, e.g., \cite{acz}) that a mapping $Q : \mathbb{R} \to  \mathcal{V}$ is quadratic if and only there exists $B : \mathbb{R}^2 \to  \mathcal{V}$ that
is symmetric (i.e., $B(x, y)= B(y,x)$ for all $x, y \in\mathbb{R}$) and biadditive (i.e.,
$B(x+y,  z)= B(x, z)+ B(y, z)$ for all $x, y,z \in\mathbb{R}$) such that $Q(x)=B(x,x)$ for all $x \in\mathbb{R}$.
\end{remark} 

We derive from Proposition \ref{p1}, Theorem \ref{tt2} and Remark \ref{rq1} the following corollaries.

\begin{corol}\label{cc2} Let $a,b,c,d\in\mathbb{R}\setminus\{0\}$ be fixed numbers and $\mathcal{ V}$ a real vector space. A function $f : \mathbb{R} \to  \mathcal{V} $ satisfies  \eqref{eq0}  if and only if:
\begin{itemize}
	\item [(i)] In the case $c+d\neq 2,$  there is a
symmetric biadditive  mapping $B : \mathbb{R} \to  \mathcal{V}$  such that $f(x)=B\big(x^3,x^3\big)$ for all $ x\in\mathbb{R},$ and
	\begin{equation}\label{bi}
	B(ax,ay)=\frac{c}{2}B(x,y),\quad B(bx,by)=\frac{d}{2}B(x,y),\;\;x,y\in\mathbb{R}.
	\end{equation}
	\item [(ii)]In the case $c+d= 2,$  there are $w \in\mathcal{V}$ and a
symmetric biadditive  mapping $B : \mathbb{R} \to  \mathcal{V}$   such that \eqref{bi} holds and
		$f(x)=B\big(x^3,x^3\big)+w$ for all $x\in\mathbb{R}.$
\end{itemize}
\end{corol}
\begin{corol}\label{cc3} Let $a,b,c,d\in\mathbb{R}\setminus\{0\}$ be fixed numbers and $\mathcal{ V}$ a real vector space. A function $g : \mathbb{R} \to  \mathcal{V} $ satisfies \eqref{fssi}  if and only if:
\begin{itemize}
	\item [(i)] In the case $c+d\neq 2,$  there is a
symmetric biadditive  mapping $B : \mathbb{R} \to  \mathcal{V}$  such that $g(x)=B\big(x,x\big)$ and
	\begin{equation}\label{bii}
B(ax,ay)=\frac{c}{2}B(x,y),\;\;\; B(bx,by)=\frac{d}{2}B(x,y),\;\;x,y\in\mathbb{R}.
	\end{equation}
	\item [(ii)]In the case $c+d= 2,$  there are $w \in\mathcal{V}$ and a
symmetric biadditive  mapping $B : \mathbb{R} \to  \mathcal{V}$   such that \eqref{bii} holds and
		$g(x)=B\big(x,x\big)+w$  for all $x\in\mathbb{R}.$
\end{itemize}
\end{corol}

\section{ Hyperstability criterion  of Eq.~\eqref{eq000} and some consequences}\label{sol+}

By using  Theorem \ref{pl1}, we study the hyperstability results of Eq. \eqref{eq000} in quasi-$(2,\beta)$-Banach space. In the following, we consider   $(X,\|\cdot,\cdot\|_{q,\beta},\kappa)$ is  a  quasi-$(2,\beta)$-Banach space, $a,b,c,d\in\mathbb{R}\setminus\{0\}$ are fixed numbers, $\theta=\beta\log_{2\kappa}2$ and $\mathbb{N}_{m_{0}}:=\{m\in\mathbb{N}: m\ge m_0\}$ with $m_{0}\in \mathbb{N}$. A function $f : \mathbb{R}\to X$ fulfilling Eq. \eqref{eq000} $\gamma$-approximately if 
\begin{align}\label{ezd}
\Big\|f\left(\sqrt[3]{ax^{3}+by^{3}}\right)+f\left(\sqrt[3]{ax^{3}-by^{3}}\right)-cf(x)-df(y),z\Big\|_{q,\beta} \leq \gamma(x,y,z)
\end{align}
 for all $(x,y,z)\in \mathbb{R}_0^2\times X$, where $\sqrt[3]{a}x\neq \pm\sqrt[3]{b}y$ and  $\gamma : \mathbb{R}^2\times X\to \mathbb{R}_+$ is a given function.

\begin{theorem} \label{plh2}Let $ h_i: \mathbb{R}\times X \to \mathbb{R}_+$ be a given function for $i\in\{1,2,3,4\}$, and let
\begin{align}\label{eq+hm}
\mathcal{M}_0:=\Big\{n\in\mathbb{N}_2\mid P_n:=\max\{A_n,B_n,C_n\}<1\Big\}\neq\emptyset,
\end{align}
where 
$$A_n=\kappa|c|^{\beta}s_1(u_n^3)s_2(u_n^3)+\kappa^2|d|^{\beta}s_1(v_n^3)s_2(v_n^3)+\kappa^2s_1(w_n^3)s_2(w_n^3)$$
$$B_n=\kappa|c|^{\beta}s_3(u_n^3)+\kappa^2|d|^{\beta}s_3(v_n^3)+\kappa^2s_3(w_n^3)$$
$$C_n=\kappa|c|^{\beta}s_4(u_n^3)+\kappa^2|d|^{\beta}s_4(v_n^3)+\kappa^2s_4(w_n^3)$$
$$u_n=\frac{n}{\sqrt[3]{a}},\;\; v_n=\sqrt[3]{\frac{1-n^3}{b}},\;\;w_n=\sqrt[3]{2n^3-1}$$
$$\lim_{n\to\infty}\max\{s_1(u_n^3)s_2(v_n^3),s_3(u_n^3),s_4(v_n^3)\}=0$$
 and $s_i(\rho):=\inf\big\{t\in\mathbb{R}_+: h_i\big( \rho x^{3},z\big)\le th_i\big(x^{3},z\big) \;\;\text{for all}\;\; (x,z) \in\mathbb{R}\times X\big\}$ for  $\rho\in\mathbb{R}_0$ and $i\in\{1,2,3,4\}$, such that one of the conditions holds:
\begin{enumerate}
	\item [\rm{(i)}]$\lim_{ |\rho|\to+\infty}s_1( \rho)s_2( \rho)=0$ if $P_n=A_n,$
	\item [\rm{(ii)}]$\lim_{ |\rho|\to+\infty}s_3(\rho)=0$ if $P_n=B_n,$
	\item [\rm{(iii)}]$\lim_{|\rho|\to+\infty}s_4( \rho)=0$ if $P_n=C_n$.
\end{enumerate}
 If  $f : \mathbb{R}\to X$ satisfies \eqref{ezd} with $\gamma(x,y,z)=h_1\big(x^3,z\big)h_2\big(y^3,z\big)+h_3\big(x^3,z\big)+h_4\big(y^3,z\big)$, then \eqref{eq000} holds.
\end{theorem}

\begin{proof}   It is clear that for  $\rho\in\mathbb{R}_0$ and $i=\{1,2,3,4\}$, we have
\begin{align}\label{ineq} 
h_i\big( \rho x^{3},z\big)\le s_i(\rho)h_i\big(x^{3},z\big),\quad (x,z) \in\mathbb{R}\times X.
\end{align}

 Replacing $(x,y)$ by $\left(u_m x,v_m x\right)$ in \eqref{ezd}, with $u_m=\frac{m}{\sqrt[3]{a}}$ and $v_m= \sqrt[3]{\frac{1-m^3}{b}}$,  we get  
\begin{align}\label{eq4h}
\|cf(u_mx)+df(v_mx)-f(w_mx)-f(x),z\|_{q,\beta}\leq& h_1\big(u_m^3x^3,z\big)h_2\big(v_m^3x^3,z\big)\nonumber\\&+h_3\big(u_m^3x^3,z\big)+h_4\big(v_m^3x^3,z\big)
\end{align}
for all  $x\in \mathbb{R}_0$ and all $z\in X$ with $w_m=\sqrt[3]{2m^3-1}$ and $m\in\mathbb{N}_2$ .   For each $m\in\mathbb{N}_2$, we will define  an operator
$\mathcal{T}_m: X^{\mathbb{R}_0}\to X^{\mathbb{R}_0}$  by
$$(\mathcal{T}_m\xi)(x):= c\xi(u_mx)+d\xi(v_mx)-\xi(w_mx),\quad \xi\in X^{\mathbb{R}_0},\;\;x\in \mathbb{R}_0.$$
Putting $$\varepsilon_ m(x,z) :=h_1\big(u_m^3x^3,z\big)h_2\big(v_m^3x^3,z\big)\nonumber+h_3\big(u_m^3x^3,z\big)+h_4\big(v_m^3x^3,z\big)$$ for all  $x\in \mathbb{R}_0$ and $z\in X$ with $m\in\mathbb{N}_2$. Then by \eqref{ineq}  we have
\begin{align}\label{eqeh}
\varepsilon_ m(x,z)\le \sigma_m \big[h_1\big(x^3,z\big)h_2\big(x^3,z\big)+h_3\big(x^3,z\big)+h_4\big(x^3,z\big)\big],\quad\;\;(x,z)\in \mathbb{R}_0\times X,
\end{align}
with $\sigma_m:=\max\{s_1(u_m^3)s_2(v_m^3),s_3(u_m^3),s_4(v_m^3)\}.$
 Then the inequality (\ref{eq4h}) takes the form
$$\|(\mathcal{T}_mf)(x)-f(x),z\|_{q,\beta}\leq\varepsilon_m(x,z), \quad  (x,z)\in \mathbb{R}_0\times X,\;\;m\in\mathbb{N}_2.$$

For each $m\in\mathbb{N}_2$, we find that the operator $\Lambda_m:\mathbb{R}_+^{\mathbb{R}_0\times X}\to \mathbb{R}_+^{\mathbb{R}_0}\times X$ defined by
$$(\Lambda_m\delta)(x,z)=\kappa |c|^{\beta}\delta\left(u_mx,z\right)+\kappa^2|d|^{\beta}\delta(v_mx,z)+\kappa^2\delta(w_mx,z),\quad\delta\in \mathbb{R}_+^{\mathbb{R}_0\times X},\;\;x\in \mathbb{R}_0$$
has the shape given in (\textbf{A4}) with $j=3$, 
$$f_{1}(x)\equiv u_mx,\quad f_{2}(x)\equiv v_mx,\quad f_{3}(x)\equiv w_mx, \quad L_{1}(x,z)\equiv\kappa |c|^{\beta},$$$$ \quad L_{2}(x,z)\equiv \kappa^2|d|^{\beta},\quad L_{3}(x,z)\equiv \kappa^2;\;\;(x,z)\in\mathbb{R}_0\times X,\;\;m\in\mathbb{N}_2.$$ 
  Moreover, for every $\xi, \mu\in X^{\mathbb{R}_0}$, $m\in\mathbb{N}_{2}$  and $(x,z)\in \mathbb{R}_0\times X$, we obtain
\begin{align*}
\big\|(\mathcal{T}_m\xi)(x)&-(\mathcal{T}_m\mu)(x),z\big\|_{q,\beta}\\&=
\left\|c\xi(u_mx)+d\xi(v_mx)-\xi(w_mx) -c\mu(u_mx)-d\mu(v_mx)+\mu(w_mx),z\right\|_{q,\beta}
\\&\le
  \kappa |c|^{\beta}\left\|\xi(f_{1}(x))-\mu(f_{1}(x)),z\right\|_{q,\beta}+\kappa^2 |d|^{\beta}\left\|\xi(f_{2}(x))-\mu(f_{2}(x)),z\right\|_{q,\beta}\\&+\kappa^2 \left\|\xi(f_{3}(x))-\mu(f_{3}(x)),z\right\|_{q,\beta}
\\&= \sum_{i=1}^{3}L_i(x,z)\big\|\xi(f_{i}(x))-\mu(f_{i}(x)),z\big\|_{q,\beta}.
\end{align*}
So, \textbf{(A3)} is valid  for $\mathcal{T}_m$ with $m\in\mathbb{N}_2$.  It is not hard to show that
\begin{align}\label{chk}
\Lambda_m\varepsilon_{m}(x,z)\le P_m \sigma_m \big[h_1\big(x^3,z\big)h_2\big(x^3,z\big)+h_3\big(x^3,z\big)+h_4\big(x^3,z\big)\big]
\end{align}
for all $(x,z)\in \mathbb{R}_0\times X.$ By  induction, we will show that for each $n\in \mathbb{N}_0$ and $(x,z)\in \mathbb{R}_0\times X,$   
\begin{align}\label{mlh}
\Lambda_m^{n}\varepsilon_{m}(x,z)\le P_m^n \sigma_m \big[h_1\big(x^3,z\big)h_2\big(x^3,z\big)+h_3\big(x^3,z\big)+h_4\big(x^3,z\big)\big]
\end{align}
where $m\in\mathcal{M}_0$. From (\ref{eqeh}), we obtain that the inequality  (\ref{mlh}) holds for $n = 0.$ Next, we will assume that (\ref{mlh}) holds for $n = r$, where $r \in \mathbb{N}_0$. Then,
\begin{align*}
\Lambda_m^{r+1}&\varepsilon_{m}(x,z)=\Lambda_m(\Lambda_m^{r}\varepsilon_{m}(x,z))
\\=&\;\kappa |c|^{\beta}\Lambda_m^{r}\varepsilon_{m}\left(u_mx,z\right)+\kappa^2|d|^{\beta}\Lambda_m^{r}\varepsilon_{m}(v_mx,z)+\kappa^2\Lambda_m^{r}\varepsilon_{m}(w_mx,z)
\\&\le P_m^r \sigma_m \kappa |c|^{\beta}\big[h_1\big(u_m^3x^3,z\big)h_2\big(u_m^3x^3,z\big)+h_3\big(u_m^3x^3,z\big)+h_4\big(u_m^3x^3,z\big)\big]
\\&+P_m^r \sigma_m \kappa^2 |d|^{\beta}\big[h_1\big(v_m^3x^3,z\big)h_2\big(v_m^3x^3,z\big)+h_3\big(v_m^3x^3,z\big)+h_4\big(v_m^3x^3,z\big)\big] 
\\&+P_m^r \sigma_m \kappa^2 \big[h_1\big(w_m^3x^3,z\big)h_2\big(w_m^3x^3,z\big)+h_3\big(w_m^3x^3,z\big)+h_4\big(w_m^3x^3,z\big)\big]
\\&\le P_m^r \sigma_m \big[ A_m h_1\big(x^3,z\big)h_2\big(x^3,z\big)+B_m h_3\big(x^3,z\big)+C_m h_4\big(x^3,z\big)\big]
\\&\le P_m^r \sigma_m\max\{A_m,B_m,C_m\}\big[h_1\big(x^3,z\big)h_2\big(x^3,z\big)+h_3\big(x^3,z\big)+h_4\big(x^3,z\big)\big]
\\&=P_m^{r+1} \sigma_m\big[h_1\big(x^3,z\big)h_2\big(x^3,z\big)+h_3\big(x^3,z\big)+h_4\big(x^3,z\big)\big]
\end{align*}
for all $(x,z)\in \mathbb{R}_0\times X$ and $m\in\mathcal{M}_0$. This shows that (\ref{mlh}) holds for all $n\in \mathbb{N}_0$. 

 By  the definition of $\mathcal{M}_0$, we find that for each $(x,z)\in \mathbb{R}_0\times X$ and $m\in\mathcal{M}_0$,
\begin{align*}
\varepsilon^{*}_m(x,z)=\sum_{n=0}^{\infty}(\Lambda_m^{n}\varepsilon_{m})^{\theta}(x,z)
\le& \sum_{n=0}^{\infty}P_m^{n\theta} \sigma_m^{\theta} \big[h_1\big(x^3,z\big)h_2\big(x^3,z\big)+h_3\big(x^3,z\big)+h_4\big(x^3,z\big)\big]^{\theta}
\\=& \frac{\sigma_m^{\theta} \big[h_1\big(x^3,z\big)h_2\big(x^3,z\big)+h_3\big(x^3,z\big)+h_4\big(x^3,z\big)\big]^{\theta}}{1-P_m^{\theta }}.
\end{align*}
Thus, according to Theorem \ref{pl1}, there exists a fixed point $ Q_{m}: \mathbb{R}_0\to X$ of the operator $\mathcal{T}_m$ satisfying
\begin{align}\label{theta}
\big\|f(x)-Q_{m}(x),z\big\|_{q,\beta}^{\theta}&\le K\varepsilon^{*}_m(x,z)\nonumber\\&\leq \frac{K\sigma_m^{\theta} \big[h_1\big(x^3,z\big)h_2\big(x^3,z\big)+h_3\big(x^3,z\big)+h_4\big(x^3,z\big)\big]^{\theta}}{1-P_m^{\theta }}
\end{align}
for all $(x,z)\in \mathbb{R}_0\times X,$ $m\in\mathcal{M}_0$ and for some constant $K>0$. That is,
$$ Q_{m}(x)= cQ_{m}(u_mx)+dQ_{m}(v_mx)-Q_{m}(w_mx),\quad x\in \mathbb{R}_0,\;\;m\in\mathcal{M}_0,$$ 
and  \eqref{theta} holds for all $(x,z)\in \mathbb{R}_0\times X$. Moreover;
\begin{align}\label{limit}
Q_m(x):=\lim_{n\to\infty}\mathcal{T}_m^{n}f(x).
\end{align}

Now, we show that for every $n \in\mathbb{N}_0$ and $(x, y,z)\in \mathbb{R}_0^2\times X$ with $\sqrt[3]{a}x\neq \pm\sqrt[3]{b}y$,
\begin{align}\label{eq5h} 
\Big\|\mathcal{T}_m^{n}f\left(\sqrt[3]{ax^3+by^3}\right)&+\mathcal{T}_m^{n}f\left(\sqrt[3]{ax^3-by^3}\right)-c\mathcal{T}_m^{n}f(x)-d\mathcal{T}_m^{n}f(y),z\Big\|_{q,\beta}\\&\leq P_m^{n} \big[h_1\big(x^3,z\big)h_2\big(y^3,z\big)+h_3\big(x^3,z\big)+h_4\big(y^3,z\big)\big]\nonumber
\end{align} 
where $m \in\mathcal{M}_0$ and $\gamma(x,y,z):=\big[h_1\big(x^3,z\big)h_2\big(y^3,z\big)+h_3\big(x^3,z\big)+h_4\big(y^3,z\big)\big]$.

Indeed, if $n =0$, then (\ref{eq5h}) is simply (\ref{ezd}). So, fix $n\in\mathbb{N}_0$ and assume that (\ref{eq5h}) holds for $n$ and every  $(x, y,z) \in \mathbb{R}_0^2\times X$ with $\sqrt[3]{a}x\neq \pm\sqrt[3]{b}y$. Then, by the definition of $\mathcal{T}_m$ and \eqref{ineq}, we get that
\begin{align*}
\Big\|&\mathcal{T}_m^{n+1}f\left(\sqrt[3]{ax^3+by^3}\right)+\mathcal{T}_m^{n+1}f\left(\sqrt[3]{ax^3-by^3}\right)-c\mathcal{T}_m^{n+1}f(x)-d\mathcal{T}_m^{n+1}f(y),z\Big\|_{q,\beta}
\\&=\;\Big\|c\mathcal{T}_m^{n}f\left(u_m\sqrt[3]{ax^3+by^3}\right)+d\mathcal{T}_m^{n}f\left(v_m\sqrt[3]{ax^3+by^3}\right)-\mathcal{T}_m^{n}f\left(w_m\sqrt[3]{ax^3+by^3}\right)\\&+
c\mathcal{T}_m^{n}f\left(u_m\sqrt[3]{ax^3-by^3}\right)+d\mathcal{T}_m^{n}f\left(v_m\sqrt[3]{ax^3-by^3}\right)-\mathcal{T}_m^{n}f\left(w_m\sqrt[3]{ax^3-by^3}\right)
\\&-c\big[c\mathcal{T}_m^{n}f\left(u_mx\right)+d\mathcal{T}_m^{n}f\left(v_mx\right)-\mathcal{T}_m^{n}f\left(w_mx\right)\big]
-d\big[c\mathcal{T}_m^{n}f\left(u_my\right)+d\mathcal{T}_m^{n}f\left(v_my\right)-\mathcal{T}_m^{n}f\left(w_my\right)\big]
,z\Big\|_{q,\beta}
\\&\le\;\; \kappa|c|^{\beta}\Big\|\mathcal{T}_m^{n}f\left(u_m\sqrt[3]{ax^3+by^3}\right)+\mathcal{T}_m^{n}f\left(u_m\sqrt[3]{ax^3-by^3}\right)-c\mathcal{T}_m^{n}f\left(u_mx\right)-d\mathcal{T}_m^{n}f\left(u_my\right),z\Big\|_{q,\beta}
\\&+\kappa^2|d|^{\beta}\Big\|\mathcal{T}_m^{n}f\left(v_m\sqrt[3]{ax^3+by^3}\right)+\mathcal{T}_m^{n}f\left(v_m\sqrt[3]{ax^3-by^3}\right)-c\mathcal{T}_m^{n}f\left(v_mx\right)-d\mathcal{T}_m^{n}f\left(v_my\right),z\Big\|_{q,\beta}
\\&+\kappa^2\Big\|\mathcal{T}_m^{n}f\left(w_m\sqrt[3]{ax^3+by^3}\right)+\mathcal{T}_m^{n}f\left(w_m\sqrt[3]{ax^3-by^3}\right)-c\mathcal{T}_m^{n}f\left(w_mx\right)-d\mathcal{T}_m^{n}f\left(v_my\right),z\Big\|_{q,\beta}
\\&\le P_m^n\kappa|c|^{\beta}\big\{h_1\big(u_m^3x^3,z\big)h_2\big(u_m^3y^3,z\big)+h_3\big(u_m^3x^3,z\big)+h_4\big(u_m^3y^3,z\big)\big\}
\\&+P_m^n\kappa^2|d|^{\beta}\big\{h_1\big(v_m^3x^3,z\big)h_2\big(v_m^3y^3,z\big)+h_3\big(v_m^3x^3,z\big)+h_4\big(v_m^3y^3,z\big)\big\}
\\&+P_m^n\kappa^2\big\{h_1\big(w_m^3x^3,z\big)h_2\big(w_m^3y^3,z\big)+h_3\big(w_m^3x^3,z\big)+h_4\big(w_m^3y^3,z\big)\big\}
\\&\le P_m^n \max\{A_m,B_m,C_m\}\big[h_1\big(x^3,z\big)h_2\big(y^3,z\big)+h_3\big(x^3,z\big)+h_4\big(y^3,z\big)\big\}
\\&= P_m^{n+1}\big\{h_1\big(x^3,z\big)h_2\big(y^3,z\big)+h_3\big(x^3,z\big)+h_4\big(y^3,z\big)\big\}
\end{align*}
for all $(x,y,z)\in \mathbb{R}_0^2\times X$ with $\sqrt[3]{a}x\neq \pm\sqrt[3]{b}y$. Thus, by induction, we have shown that (\ref{eq5h}) holds for all $n\in\mathbb{N}_0$ and for all $(x,y,z)\in \mathbb{R}_0^2\times X$ with $\sqrt[3]{a}x\neq \pm\sqrt[3]{b}y$. 

 From \eqref{eq5h} and \eqref{eq} we find that
\begin{align}\label{Dd} 
\Big\|\Big|\mathcal{T}_m^{n}&f\left(\sqrt[3]{ax^3+by^3}\right)+\mathcal{T}_m^{n}f\left(\sqrt[3]{ax^3-by^3}\right)-c\mathcal{T}_m^{n}f(x)-d\mathcal{T}_m^{n}f(y),z\Big|\Big\|_{q,\beta}^{\theta}\nonumber\\&\leq
C_2^{\theta}\Big\|\mathcal{T}_m^{n}f\left(\sqrt[3]{ax^3+by^3}\right)+\mathcal{T}_m^{n}f\left(\sqrt[3]{ax^3-by^3}\right)-c\mathcal{T}_m^{n}f(x)-d\mathcal{T}_m^{n}f(y),z\Big\|_{q,\beta}^{\theta}\nonumber\\&\leq C_2^{\theta}P_m^{n\theta} \big[h_1\big(x^3,z\big)h_2\big(y^3,z\big)+h_3\big(x^3,z\big)+h_4\big(y^3,z\big)\big]^{\theta}
\end{align}
for some constant $C_2>0$. We conclude from the continuity of $\||\cdot,\cdot|\|^{\theta},$   \eqref{limit} and \eqref{Dd} that, for every $(x,y,z)\in \mathbb{R}_0^2\times X$ with $\sqrt[3]{a}x\neq \pm\sqrt[3]{b}y,$
\begin{align}\label{Dd+} 
\Big\|\Big|&Q_m\left(\sqrt[3]{ax^3+by^3}\right)+Q_m\left(\sqrt[3]{ax^3-by^3}\right)-cQ_m(x)-dQ_m(y),z\Big|\Big\|_{q,\beta}^{\theta}\\&=\lim_{n\to\infty}\Big\|\Big|\mathcal{T}_m^{n}f\left(\sqrt[3]{ax^3+by^3}\right)+\mathcal{T}_m^{n}f\left(\sqrt[3]{ax^3-by^3}\right)-c\mathcal{T}_m^{n}f(x)-d\mathcal{T}_m^{n}f(y),z\Big|\Big\|_{q,\beta}^{\theta}\nonumber\\&=0,\nonumber
\end{align}
this means
\begin{align}\label{kwr} 
Q_m\left(\sqrt[3]{ax^3+by^3}\right)+Q_m\left(\sqrt[3]{ax^3-by^3}\right)=cQ_m(x)+dQ_m(y)
\end{align}
for all $x,y\in \mathbb{R}_0$, with $\sqrt[3]{a}x\neq \pm\sqrt[3]{b}y.$

We want now to prove that the mapping $Q_m:\mathbb{R}_0\to X$  is  unique. So, let $G_m:\mathbb{R}_0\to X$ be a solution of \eqref{kwr} and
\begin{align}\label{theta0}
\big\|f(x)-G_{m}(x),z\big\|_{q,\beta}^{\theta}\leq \frac{K\sigma_m^{\theta} \big[h_1\big(x^3,z\big)h_2\big(x^3,z\big)+h_3\big(x^3,z\big)+h_4\big(x^3,z\big)\big]^{\theta}}{1-P_m^{\theta }}
\end{align}
for all $(x,z)\in \mathbb{R}_0\times X.$ Thus, replacing $(x,y)$ by $(u_mx,v_my)$  in \eqref{kwr}, we get $\mathcal{T}_mG_m(x)=G_m(x)$ for all $x\in \mathbb{R}_0$ and $m\in\mathcal{M}_0$.

Moreover, we have
\begin{align}\label{theta+0}
\||Q_m(x)-G_{m}(x),z|\|_{q,\beta}&\le C_2\|Q_m(x)-G_{m}(x),z\|_{q,\beta}\nonumber
\\&\le \kappa C_2\big(\|Q_m(x)-f(x),z\|_{q,\beta}+\|G_m(x)-f(x),z\|_{q,\beta}\big)\nonumber
\\&\le\;L\big[h_1\big(x^3,z\big)h_2\big(x^3,z\big)+h_3\big(x^3,z\big)+h_4\big(x^3,z\big)\big],
\end{align}
where $L:=\frac{2\kappa C_2 K^{1/{\theta}}\sigma_m}{(1-P_m^{\theta})^{1/{\theta}}}$ and $m\in \mathcal{M}_0.$   It is easy to prove by induction on $n$ that 
\begin{align}\label{theta+1}
\||Q_m(x)-G_{m}(x),z|\|_{q,\beta}\le \;LP_m^{n}\big[h_1\big(x^3,z\big)h_2\big(x^3,z\big)+h_3\big(x^3,z\big)+h_4\big(x^3,z\big)\big]
\end{align}
for all $x\in \mathbb{R}_0$ and $m\in\mathcal{M}_0$. Letting $n\to\infty$ in (\ref{theta+1}), we get  $Q_m\equiv G_m$. So, the fixed point satisfying \eqref{theta} of $\mathcal{T}_m$  is unique.

Letting $m\to\infty$ in (\ref{theta}), we obtain $\lim_{m\to\infty}\|f(x)-Q_m(x),z\|_{q,\beta}=0$, i.e.,
\begin{align}\label{Bmh}
\lim_{m\to\infty}Q_m(x)=f(x).
\end{align}
Also, letting $m \to\infty$ in \eqref{Dd+}, using \eqref{Bmh}  and the continuity of $\||\cdot,\cdot|\|$, we have
\begin{align*}
\Big\|\Big|&f\left(\sqrt[3]{ax^3+by^3}\right)+f\left(\sqrt[3]{ax^3-by^3}\right)-cf(x)-df(y),z\Big|\Big\|_{q,\beta}\nonumber\\&=\lim_{m\to\infty}\Big\|\Big|Q_m\left(\sqrt[3]{ax^3+by^3}\right)+Q_m\left(\sqrt[3]{ax^3-by^3}\right)-cQ_m(x)-dQ_m(y),z\Big|\Big\|_{q,\beta}=0
\end{align*}
for all $(x,y,z)\in \mathbb{R}_0^2\times X$, with $\sqrt[3]{a}x\neq \pm\sqrt[3]{b}y$, that is,
$$
f\left(\sqrt[3]{ax^3+by^3}\right)+f\left(\sqrt[3]{ax^3-by^3}\right)=cf(x)+df(y)
$$
for all $x,y\in \mathbb{R}_0$, with $\sqrt[3]{a}x\neq \pm\sqrt[3]{b}y.$ This proves that $f$ is a solution of the  radical-type functional  equation \eqref{eq000}. The proof of the theorem is complete.
\end{proof}

\begin{remark} Theorem \ref{plh2}   also provide hyperstability outcomes in each of the following cases:
\begin{enumerate}
	\item[(i)] $\gamma(x,y,z)=h_1\big(x^3,z\big)h_2\big(y^3,z\big)+h_3\big(x^3,z\big)$
		\item[(ii)] $\gamma(x,y,z)=h_1\big(x^3,z\big)h_2\big(y^3,z\big)+h_4\big(y^3,z\big)$
		\item[(iii)] $\gamma(x,y,z)=h_1\big(x^3,z\big)h_2\big(y^3,z\big)$
		\item[(iv)] $\gamma(x,y,z)=h_3\big(x^3,z\big)+h_4\big(y^3,z\big)$
		\item[(v)] $\gamma(x,y,z)=h_3\big(x^3,z\big)+h_3\big(y^3,z\big)$
		\item[(vi)] $\gamma(x,y,z)=h_1\big(x^3,z\big)$
		\item[(vii)] $\gamma(x,y,z)=h_2\big(y^3,z\big)$
\end{enumerate}
for all $(x,y,z)\in \mathbb{R}_0^2\times X.$
\end{remark}
Applying Theorem \ref{plh2} and the same technique, we get  the following corollary.

\begin{corol}\label{c2}   Under the hypotheses of Theorem \ref{plh2}, we consider two functions  $f:\mathbb{R}\to X$ and $F:\mathbb{R}^2\to X$  such that
\begin{align*}
\Big\|f\left(\sqrt[3]{ax^3+by^3}\right)+f\left(\sqrt[3]{ax^3-by^3}\right)-cf(x)-df(y)-F(x,y),z\Big\|_{q,\beta}\le \gamma(x,y,z)
\end{align*}
for all $(x,y,z)\in \mathbb{R}_0^2\times X$, with $\sqrt[3]{a}x\neq \pm\sqrt[3]{b}y$ and $\gamma(x,y,z)=h_1\big(x^3,z\big)h_2\big(y^3,z\big)+h_3\big(x^3,z\big)+h_4\big(y^3,z\big)$. Assume that the functional equation
\begin{align}\label{eq14}
\tilde{f}\left(\sqrt[3]{ax^3+by^3}\right)+\tilde{f}\left(\sqrt[3]{ax^3-by^3}\right)&=c\tilde{f}(x)+d\tilde{f}(y)+F(x,y),\\&\quad\;\;x,y\in \mathbb{R}_0,\;\;\sqrt[3]{a}x\neq \pm\sqrt[3]{b}y\nonumber
\end{align} 
admits a solution $f_0:\mathbb{R}\to X$. Then $f$ is a solution of \eqref{eq14}.
\end{corol}

\begin{proof} Let $g(x):=f(x)-f_0(x)$ for $x\in \mathbb{R}_0$. Then 
\begin{align*}
\Big\|g\Big(\sqrt[3]{ax^3+by^3}\Big)&+g\Big(\sqrt[3]{ax^3-by^3}\Big)-cg(x)-dg(y),z\Big\|_{q,\beta} 
\\\le& \; \kappa\Big\|f\Big(\sqrt[3]{ax^3+by^3}\Big)+f\Big(\sqrt[3]{ax^3-by^3}\Big)-cf(x)-df(y)-F(x,y),z\Big\|_{q,\beta}
\\&\;\;+\kappa\Big\|f_0\Big(\sqrt[3]{ax^3+by^3}\Big)+f_0\Big(\sqrt[3]{ax^3+by^3}\Big)-cf_0(x)-df_0(y)-F(x,y),z\Big\|_{q,\beta}
\\=& \;\kappa\Big\|f\Big(\sqrt[3]{ax^3+by^3}\Big)+f\Big(\sqrt[3]{ax^3-by^3}\Big)-cf(x)-df(y)-F(x,y),z\Big\|_{q,\beta}
\\ \leq&\;  \kappa\gamma(x,y,z),\;\; (x,y,z)\in \mathbb{R}_0^2\times X,\;\;\; \sqrt[3]{a}x\neq \pm\sqrt[3]{b}y.
\end{align*}
It follows from Theorem \ref{plh2}    that $g$ satisfies the  equation \eqref{eq000}.  Therefore,
\begin{align*}
f\left(\sqrt[3]{ax^3+by^3}\right)&+f\left(\sqrt[3]{ax^3-by^3}\right)-cf(x)-df(y)-F(x,y)
\\&=g\Big(\sqrt[3]{ax^3+by^3}\Big)+g\Big(\sqrt[3]{ax^3-by^3}\Big)-cg(x)-dg(y)
\\&+f_0\Big(\sqrt[3]{ax^3+by^3}\Big)+f_0\Big(\sqrt[3]{ax^3-by^3}\Big)-cf_0(x)-df_0(y)-F(x,y)
=0
\end{align*}
for all $x, y\in \mathbb{R}$ with $\sqrt[5]{a}x\neq -\sqrt[5]{b}y$.
\end{proof}

To end this section, we derive some particular cases from Theorem \ref{plh2} and Corollary \ref{c2}.\\
Let $(X,\|\cdot,\cdot\|_{q,\beta},\kappa)$ be a quasi-$(2,\beta)$-Banach space,  $(Y,\|\cdot,\cdot\|_{q,\alpha},\kappa')$  a quasi-$(2,\alpha)$-normed space. Also, let $\lambda_i:\mathbb{R}\to Y\setminus\{0\}$ be additive functions continuous at a point for $i=1,2,3,4,$ and  $g:X\to Y\setminus\{0\}$ a surjective mapping. According to Theorem \ref{plh2} and Corollary \ref{c2},  we give some consequences with  $$h_i(x,z) :=c_i\|\lambda_i(x),g(z)\|_{q,\alpha}^{p_i},\qquad (x,z)\in\mathbb{R}_0\times X,$$ 
 where  $c_i,p_i\in\mathbb{R}$ for $i\in\{1,2,3,4\}.$

\begin{corol}  Let $a,b,c,d\in\mathbb{R}_0$ be fixed numbers, $D_1,D_2,D_3\geq 0$ be three constants and let $p_i\in \mathbb{R}$ such that $p_1+p_2, p_3,p_4<0$ for $i\in\{1,2,3,4\}.$ If  $f : \mathbb{R} \to X$ satisfies the inequality
\begin{align*}
\Big\|f\Big(&\sqrt[3]{ax^3+by^3}\Big)+f\left(\sqrt[3]{ax^3-by^3}\right)-cf(x)-df(y),z\Big\|_{q,\beta} \\&\leq D_1\left\|\lambda_1(x^3),g(z)\right\|_{q,\alpha}^{p_1}\left\|\lambda_2(y^3),g(z)\right\|_{q,\alpha}^{p_2}+D_2\left\|\lambda_3(x^3),g(z)\right\|_{q,\alpha}^{p_3}+D_3\left\|\lambda_4(y^3),g(z)\right\|_{q,\alpha}^{p_4}\nonumber
\end{align*}
 for all $(x,y,z)\in \mathbb{R}_0^2\times X$ with $\sqrt[3]{a}x\neq \pm\sqrt[3]{b}y$, then \eqref{eq000} holds.
\end{corol}

\begin{proof} Define $h_i : \mathbb{R}_0\times X \to \mathbb{R}_+$ by $h_i(x^3,z)=c_i\left\|\lambda_i(x^3),g(z)\right\|_{q,\alpha}^{p_i}$  for some $c_i\in\mathbb{R}_+$ and  $p_3,p_4,p_1+p_2<0$ with $i\in\{1,2,3,4\}$ and $(D_1,D_2,D_3)= (c_1c_2, c_3, c_4)$.

For each $\rho\in \mathbb{R}_0,$ we have
\begin{align*}
s_i(\rho)&=\inf\Big\{t\in \mathbb{R}_+:h_i\big(\rho x^3,z\big)\le th_i\big(x^3,z\big),\;\;\forall (x,z)\in \mathbb{R}\times X\Big\} 
\\&=\inf\Big\{t\in \mathbb{R}_+:c_i\left\|\lambda_i(\rho x^3),g(z)\right\|_{q,\alpha}^{p_i}\le tc_i\left\|\lambda_i(x^3)g(z)\right\|_{q,\alpha}^{p_i},\;\; \forall (x,z)\in \mathbb{R}\times X\Big\}
\\&=|\rho|^{\alpha p_i}
\end{align*}
 for $i\in\{1,2,3,4\}.$ So,
\begin{align*}
\lim_{n\to\infty}&\max\{s_1(u_n^3)s_2(v_n^3),s_3(u_n^3),s_4(v_n^3)\}\\&=
\lim_{n\to\infty}\max\left\{s_1\Big(\frac{n^3}{a}\Big)s_2\Big(\frac{1-n^3}{b}\Big),s_3\Big(\frac{n^3}{a}\Big),s_4\Big(\frac{1-n^3}{b}\Big)\right\}\\&=
\lim_{n\to\infty}\max\left\{\Big|\frac{n^3}{a}\Big|^{\alpha p_1}\Big|\frac{1-n^3}{b}\Big|^{\alpha p_2},\Big|\frac{n^3}{a}\Big|^{\alpha p_3},\Big|\frac{1-n^3}{b}\Big|^{\alpha p_4}\right\}=0
\end{align*}
where $u_n=\frac{n}{\sqrt[3]{a}},\;\; v_n=\sqrt[3]{\frac{1-n^3}{b}}$ and $n\in\mathbb{N}_2.$ Similarly, we have
$$\lim_{n\to\infty}P_n=\lim_{n\to\infty}\max\{A_n,B_n,C_n\}=0$$
where $A_n,\; B_n$ and $C_n$ are defined as in Theorem \ref{plh2}. Clearly, there is $n_0\in \mathbb{N}_2$ such that
$$P_n <1,\quad n\geq n_0.$$
 Thus, all the conditions in Theorem \ref{plh2} are fulfilled. 
\end{proof}

\begin{corol} 
Let $a,b,c,d\in\mathbb{R}_0$ be fixed numbers, $D_1,D_2,D_3\geq 0$ be three constants and let $p_i\in \mathbb{R}$ such that $p_1+p_2,p_3,p_4<0$ for $i\in\{1,2,3,4\}.$ Let  $f : \mathbb{R} \to X$ and $F : \mathbb{R}^2 \to X$ be two functions such that
\begin{align*}
\Big\|f\Big(&\sqrt[3]{ax^3+by^3}\Big)+f\left(\sqrt[3]{ax^3-by^3}\right)-cf(x)-df(y)-F(x,y),z\Big\|_{q,\beta} \\&\leq D_1\left\|\lambda_1(x^3),g(z)\right\|_{q,\alpha}^{p_1}\left\|\lambda_2(y^3),g(z)\right\|_{q,\alpha}^{p_2}+D_2\left\|\lambda_3(x^3),g(z)\right\|_{q,\alpha}^{p_3}+D_3\left\|\lambda_4(y^3),g(z)\right\|_{q,\alpha}^{p_4}\nonumber
\end{align*}
 for all $(x,y,z)\in \mathbb{R}_0^2\times X$ with $\sqrt[3]{a}x\neq \pm\sqrt[3]{b}y$. Assume that the functional equation
\begin{align}\label{eq14h}
\tilde{f}\left(\sqrt[3]{ax^3+by^3}\right)+\tilde{f}\left(\sqrt[3]{ax^3-by^3}\right)&=c\tilde{f}(x)+d\tilde{f}(y)+F(x,y),\\&\quad\;\;x,y\in \mathbb{R}_0,\;\;\sqrt[3]{a}x\neq \pm\sqrt[3]{b}y\nonumber
\end{align} 
admits a solution $f_0:\mathbb{R}\to X$. Then $f$ is a solution of \eqref{eq14h}.
\end{corol}

\begin{note}
In the same way, we can find the general solution and the hyperstability results of the following functional equation:
$$f\left(\sqrt[n]{ax^{n}+by^{n}}\right)+f\left(\sqrt[n]{ax^{n}-by^{n}}\right)=cf(x)+df(y),\;\;x,y \in \mathbb{R},$$
where $a, b, c$ and $d$ are nonzero real constants and $n$ is an odd positive integer.
\end{note}
\begin{question}

Considered as a future work concerning the solution and Ulam's stability for the following functional equation:
$$f\left(\sqrt[n]{|ax^{n}+by^{n}|}\right)+f\left(\sqrt[n]{|ax^{n}-by^{n}|}\right)=cf(x)+df(y),\;\;x,y \in \mathbb{R},$$
where $a, b, c$ and $d$ are nonzero real constants and $n$ is a positive integer.

\end{question}

\end{document}